\documentclass[10pt,reqno]{amsart}
\usepackage{amsmath}
\usepackage{amsfonts}
\usepackage{mathrsfs}
\usepackage{amssymb}
\usepackage{amsthm}
\usepackage{enumitem}
\usepackage{bbm}
\usepackage{times}
\usepackage{tabularx}
\usepackage{amsbsy}
\usepackage{mathtools}
\usepackage{tikz}
\usetikzlibrary{arrows,decorations.markings}
\usetikzlibrary{matrix}
\usetikzlibrary{graphs}
\usetikzlibrary{backgrounds}
\usepackage{setspace}     \spacing{1}
\usepackage{color}

\usepackage[colorlinks=true,linkcolor=blue,citecolor=blue]{hyperref}

\theoremstyle{Theorem}
\newtheorem{theorem}{Theorem} [section]

\newtheorem{proposition}[theorem]{Proposition}

\theoremstyle{definition}
\newtheorem{definition}[theorem]{Definition}

\newtheorem{remark}[theorem]{Remark}

\theoremstyle{remark}

\newlist{enumlemma}{enumerate}{3}
\setlist[enumlemma]{label*={(\alph*)}, ref= {(\alph*)} }

\usepackage{marginnote}

\newcommand{\diff}{\mathrm{Diff}}

\renewcommand{\epsilon}{\varepsilon}

\DeclareMathOperator{\Diff}{Diff}

\newcommand{\R}{\mathbb {R}}

\newcommand{\Z}{\mathbb {Z}}

\newcommand{\e}{\epsilon}

\title{Examples of diffeomorphism group cocycles with no periodic approximation}

\author[Sebastian Hurtado]{Sebastian Hurtado}
\address{University of Chicago, Chicago, IL 60637, USA}
\email{shurtados@uchicago.edu}

\long\def\symbolfootnote[#1]#2{\begingroup\def\thefootnote{\fnsymbol{footnote}}
\footnote[#1]{#2}\endgroup}

\begin{document}

\begin{abstract}
 
We construct a finitely generated subgroup of $\diff^{\infty}(\mathbb{S}^3 \times \mathbb{S}^1)$ where every element is conjugate to an isometry but such that the group action itself is far from isometric (the group has ``exponential growth of derivatives"). As a corollary, one obtains a locally constant $\diff^{\infty}(\mathbb{S}^3 \times \mathbb{S}^1)$ valued cocycle over a hyperbolic dynamical system which has elliptic behavior over its periodic orbits but which preserves a measure with non-zero top Fiber Lyapunov exponent. Additionally, we provide new examples of Banach cocycles not satisfying the periodic approximation property as first shown in \cite{KalSad}.

\end{abstract}

%\symbolfootnote[0]{\it Preliminary version.  Last updated: \today}
\maketitle

\begin{section}{Introduction}

For a hyperbolic dynamical system $f: X \to X$ (e.g. Anosov diffeomorphisms, shift on $k$-symbols, etc.) to understand when some property holds it is often enough to check the property over the periodic orbits of $f$, for example, certain Anosov diffeomorphisms preserve a $C^0$-volume form if and only if the Jacobian satisfies $|\text{Jac}_p(f^n)| = 1$ for every periodic point $p \in X$ of period $n$. This is partly a consequence of the fact that any $f$-invariant ergodic probability measure on $X$ can be approximated by a measure supported on a periodic orbit, which comes from the Anosov closing Lemma, see \cite{KatHas}, Sec. 19.2.\\

In the context of group cocycles over hyperbolic systems similar facts are known to be true. Recall that a $G$-cocycle is a continuous function $\mathcal{A}: X \to G$ where $G$ is a topological group. Kalinin showed that for a linear cocycle  (meaning $G = GL_n(\R)$) over a hyperbolic dynamical system the following holds: For any $f$-invariant ergodic probability measure $\mu$ on $X$ with corresponding $\mathcal{A}$-Lyapunov exponents $\chi^{\mu}_1 \leq \chi^{\mu}_2 \leq \dots \leq \chi^{\mu}_n$, there is an $f$-invariant probability measure $\nu$ supported over a periodic orbit in $X$ whose corresponding Lyapunov exponents $\chi^{\nu}_1 \leq \chi^{\nu}_2 \leq  \dots \leq \chi^{\nu}_m$ satisfy $|\chi^{\mu}_i - \chi^{\nu}_i| < \e$ for every $i$. See \cite{Kal} for the definitions of the previous terms and a more thorough discussion about cocycles and Lyapunov exponents.\\  

Whether similar periodic approximation properties hold for more general groups is not well understood. For example, in the case when $G$ is the group of bounded operators of a Hilbert space, Kalinin-Sadovskaya (\cite{KalSad}) showed that such periodic approximation properties for the top Lyapunov exponent can fail.\\

Nonetheless, if $G = \diff(M)$ is the group of diffeomorphism of a low dimensional manifold $M$ (i.e. $\dim(M) = 1$ or $\dim(M) =2$ preserving volume) and Lyapunov exponents are replaced with Fiber-Lyapunov exponents, Kocsard-Potrie \cite{KocPot} showed a periodic approximation property similar to the one Kalinin proved holds.\\

In the context of smooth group actions on manifolds (where my own interest come from), diffeomorphism group cocycles appear naturally. For example, If $S = \{a_1,a_2,..a_m\}$  is generating set of a subgroup $\Gamma$ of $\diff(M)$, the shift map $\sigma: S^{\Z} \to S^{\Z}$ given by $\sigma((b_i)_{i \in \Z}) = (b_{i+1})_{i \in \Z}$ is a hyperbolic dynamical system and there is a canonical locally-constant $\diff(M)$-cocycle over $\sigma$ given by $\mathcal{A}(...a_{i_{-1}}a_{i_{0}}a_{i_{1}}...) = a_{i_{0}}$.\\

The periodic approximation property by Kocsard-Potrie has the following elementary consequence, see \cite{KocPot}, \cite{Hur}:

\begin{proposition}
If $\Gamma = \langle S \rangle$ is  a subgroup of $\Diff(\mathbb{S}^1)$ (or $\Diff_{\text{vol}}(\mathbb{S}^2)$) and there exists $\lambda > 0$ and a sequence of elements $w_n \in \Gamma$ such that $\|D(w_n)\| \geq e^{\lambda |w_n|_S}$, there exists $g \in \Gamma$ having a hyperbolic fixed point on $\mathbb{S}^1$ (or $\mathbb{S}^2$).
\end{proposition}

\begin{remark}
The notation $|g|_S$ denotes word length of $g$ with respect to the generating set $S$. For a fixed choice of Riemmanian metric on $M$, the norm $\|D(g)\|$ denotes the supremum over $x \in M$ of the norms of derivative maps $D_x(g)$ and $D_x(g^{-1})$.
\end{remark}

The purpose of this note is to show that such periodic approximation property (and the previous proposition) does no longer hold if $\dim(M) \geq 4$, to state the main theorem we make the following definition:\\

\begin{definition}
 A finitely generated subgroup $\Gamma = \langle S \rangle$ of $\diff(M)$ has sub-exponential growth of derivatives if for any $\e>0$, 
there exists $C_\e$ such that for any element $g \in \Gamma$: $$\|D(g)\| \leq C_{\e}e^{\e |g|_S}$$

If this does not hold, we say $\Gamma$ has exponential growth of derivatives.

\end{definition}

\begin{remark}
The condition of having sub-exponential growth of derivatives does not depend on the generating set. 
If $g$ is a diffeomorphism, we say that $g$ has sub-exponential growth of derivatives if the cyclic group $\langle g \rangle$ has sub-exponential growth of derivatives. 
\end{remark}

This concept was key in the work of the the author in the Burnside problem on diffeomorphism groups \cite{Hur} and the work of Brown-Fisher-Hurtado \cite{BFH} in the Zimmer program. Our main result is the following:

\begin{theorem}\label{main2} There exists a finitely generated group $\Gamma \subset \diff(M)$ on the 4-dimensional manifold $ M = \mathbb{S}^3\times \mathbb{S}^1 $  such that:
\begin{enumerate}
\item The group $\Gamma$ has exponential growth of derivatives.
\item Every element $g \in G$ preserves a Riemannian metric $m_g$ on $M$ (which must depend on $g$).
\end{enumerate}
\end{theorem}

\begin{remark} To prove the existence of a Riemannian metric $m_g$, it is sufficient to show that for every $g \in G$, the sequence $\{\|Dg^n\|\}$ is bounded independent of $n$. The existence of $m_g$ then follows by averaging an arbitrary metric.
\end{remark}

As an immediate consequence, we obtain the following:

\begin{theorem}\label{main} There exists a locally constant cocycle  $\mathcal{A}: \Sigma \to \diff(M)$ such that the suspension $F: \Sigma \times M \to \Sigma \times M$ preserves a 
probability measure $\mu$ with non-zero top Fiber-Lyapunov exponent and such that every $F$-invariant measure $\nu$ supported over a periodic orbit of $\Sigma$ has zero fiber Lyapunov exponents.
\end{theorem}

Additionally, we will construct in Section \ref{remarks} new examples of cocycles in the group of bounded operators of Hilbert spaces not satisfying the periodic approximation property defined in \cite{KalSad}, the examples are a consequence of the existence of Burnside groups. 

\subsection{Acknowledgements:} I want to thank Alejandro Kocsard and Clark Butler for helpful comments and their interest in this work.
\end{section}

\begin{section}{The group}

Let $M = \mathbb{S}^3\times \mathbb{S}^1$, the group $\Gamma \subset \diff(M)$ in Theorem \ref{main} is given as follows: Let $\Gamma^{*} = \langle a_1,a_2,...,a_m \rangle$ be a subgroup of $SU(2)$ which is free. As $SU(2)$ is simple, the generic $m$-tuple of elements generates a free subgroup.\\

For every $1\leq i \leq 4$, let  $f^i_t$ be four one-parameter subgroups of $\diff(\mathbb{S}^1)$ with disjoint support (and so the $f^i_t$'s commute) such that  each $f_t^i$ has exactly one fixed hyperbolic 
repelling point $p_i$ satisfying $(f_i^1)'(p_i) = 2$. Define $F$ to be the abelian group homomorphism  $F: \R^4 \to \diff(\mathbb{S}^1)$ given by $F(x,y,z,u) = f_x^1f_y^2f_z^3f_u^4$.\\

Consider the subgroup $\Gamma = \langle A_1, A_2,...,A_m \rangle$ of $\diff(M)$ generated by the diffeomorphisms  
$A_i: \mathbb{S}^3\times \mathbb{S}^1  \to \mathbb{S}^3\times \mathbb{S}^1$ given by $$A_i(v,s) = (a_i(v), F(v)s)$$

Observe that $$A_i^{-1}(v,s) = (a_i^{-1}(v), F(-a_i^{-1}(v))s)$$\\

We will begin by proving that an element $W$ in $\Gamma$ acts isometrically, as mentioned before this follows from the following proposition by averaging an arbitrary metric on $M$ by $W$.

\begin{proposition}Let $W$ be any element of $\Gamma$. There is a constant $C_{W}$  such that $$\|D(W^n)\| \leq C_{W}$$ for every $n \in \Z$.

\end{proposition}

\begin{proof}
For sake of simplicity assume that $W = A_{i_{l}}A_{i_{l-1}}....A_{i_2}A_{i_1}$ (all the occurrences of $A_{i}$'s are positive), 
the case where negative powers appear can be worked out similarly. 
Let $w =: a_{i_{l}}a_{i_{l-1}}....a_{i_1}$ be the corresponding word in the $a_i$'s and let $a_{i_0} = \text{Id}$. An easy calculation shows that: 
$$W^n(v,s) = (w^n(v),  F(v_{n,w})s) $$  where $$v_{n,w} = \sum_{j=0}^{l-1} \sum_{k=0}^{n-1} a_{i_j}a_{i_{j-1}}...a_{i_0}w^k(v)$$

As $w_n$ is an isometry, to show that $\|D(W^n)\|$ is bounded independent of $n$ it is enough to show that $F(v_{n,w})$ lies 
in a compact subset of $\diff(\mathbb{S}^1)$, this is a consequence of the following fact:

\begin{proposition} For any $w \neq e \in \Gamma^{*}$  there is a constant $C_{w} > 0$ such that: $$\|\sum_{k=0}^{n-1} w^k(v)\| \leq C_{w}$$ for every $v \in \mathbb{S}^3$ and every $n \in \Z$.
\end{proposition}

\begin{proof}
It is enough to prove this when $v$ is an eigenvector of $w$. In this case if $\lambda$ is the corresponding eigenvalue we have $\sum_{k=0}^{n-1} w^k(v) = (\sum^{n-1}_{k=0} \lambda^k)v $. This is bounded independent of $n$ because $\lambda \neq 1$ and $|\lambda| =1$.
\end{proof}

\end{proof}

To finish the proof of Theorem \ref{main2} we need to show that the group $\Gamma$ itself has exponential growth of derivatives, more concretely:

\begin{proposition} There exists $\e>0$ and  a sequence of words $\{W_n\} \subset \Gamma$ in the generators $\{A_1, A_2,..., A_n\}$ such that each $W_n$ has length $n$ and $\|D(W_n)\| \geq e^{\e n}$.
\end{proposition}

\begin{proof}

The proof is based in the following elementary fact:

\begin{proposition}\label{trick} Given $v_0 \in \mathbb{S}^3$ and $\delta>0$, there exists $S= \{a_1,a_2....,a_m\} \subset SU(2)$  generating a free subgroup and  an 
infinite sequence $a_{i_1} a_{i_2}...a_{i_n}...$ in $S$ such that the words $w_1:= a_{i_1}$, $w_n := a_{i_n}w_{n-1}$ satisfy: $$\|w_n(v_0) - v_0\| \leq \delta$$ for every $n \geq 1$.
\end{proposition}

\begin{proof}

For $r > 0$, let $B_r$ be the ball in $\R^4$ with center $v_0$ and radius $r$. Consider the annulus $A_{\delta} := B_{\delta} - B_{\delta/2}$.  By a compactness argument,  there is a finite set of elements $S' \subset SU(2)$ such that for any point $p \in A$ there is $g \in S'$ such that $g(p) \in B_{\delta/2}$. Take an element $h$ of $SU(2)$  such that $h(B_{\delta/2}) \subset  B_{\delta}$. We can take $S'$ and $h$ generic so that $S :=S' \cup \{h\}$ generates a free group.\\

Our words can be constructed as follows:  Apply $h$ repeatedly to the vector $v_0$ until $h^k(v_0) \in A_{\delta}$, then apply the corresponding element in $S'$ sending back $h^k(v_0)$ inside  $B_{\delta/2}$. Iterate this procedure of going back and forth between $B_{\delta/2}$ and $A_{\delta}$ using $h$ and $S'$.

\end{proof}

Using Proposition \ref{trick} with $v_0= (1,0,0,0)$ and $\delta = \frac{1}{2}$, we obtain a generating set $S= \{a_1,a_2...a_m\}$ and words $w_n$ in $S$ for every $n\geq 1$. Let $W_n$ be the corresponding word obtained by replacing $a_i$'s by $A_i$'s, we have that:

$$W_n(v_0,s) = (w_n(v_0),  F(v_{w_n})s) $$ 
where $$v_{w_n} :=  \sum_{i=0}^{n-1} w_i(v_0)$$ 

And so one has that $\|v_{w_n} - nv_0\| \leq \frac{n}{2}$ which implies that $\|D_{p_1}(F(v_{w_n}))\| \geq |(f_1^1)'(p_1)|^{\frac{n}{2}}$ where $p_1$ is the repelling point of $f^1_t$.

\end{proof}

\end{section}

\begin{section} {Further consequences and remarks}\label{remarks}

\subsection{Banach Cocycles} From the previous construction one can obtain new examples of Banach cocycles where the periodic approximation property defined in \cite{KalSad} fail. These examples are obtained by looking at Banach spaces where $\diff(M)$ acts naturally.\\

More elementary examples of Banach cocycles where the periodic approximation property defined in \cite{KalSad} fail can be constructed from the existence of Burnside groups as follows:\\

Let $G$ be a Burnside group (an infinite group which is finitely generated and such that every element of $G$ has finite order), let $S = \{a_1,a_2,...a_m\}$ be a generating set of $G$. Consider the infinite dimensional vector space spanned by formal finite sums of elements in $G$: $$X = \{ \sum c_i g_i , \text{ where } c_i \in \R, g_i \in G \}$$ Define an inner product on $X$ by setting $\|g_i\| = 2^{|g_i|_S}$ and $\langle g_i, g_j \rangle = 0$ if  $g_i\neq g_j$. Complete $X$ to a Hilbert space $\mathcal{H}$ using this norm. There is a natural action $\alpha: G \to B(\mathcal{H})$ given by left multiplication on $X$ and one easily checks that $\|\alpha(g)\|_{B(\mathcal{H})} = 2^{|g|_S}$. The corresponding cocycle $\mathcal{A}$ over the shift in $S^{\Z}$ satisfy $\chi_{\text{top}} = \log(2)$ but  $\chi_{\mu} = 0$ for any measure $\mu$ supported over a periodic orbit.

\subsection{Open questions} 

The question whether the periodic approximation property above holds for $\diff(M)$ cocycles when $dim(M) = 2$ remains open. A positive solution will have strong consequences for actions of groups on surfaces, for example it will imply there is no Burnside group in $\diff(\mathbb{S}^2)$ by using the techniques developed in \cite{Hur}.

\end{section}

\end{document}